\documentclass[11pt]{article}

\usepackage[T2A]{fontenc}
\usepackage[cp1251]{inputenc}
\usepackage{amsthm}
\usepackage{amsfonts}
\usepackage{eufrak}
\usepackage{amssymb}
\usepackage{amsmath}
\usepackage{cite}

\makeatletter
\renewcommand{\@seccntformat}[1]{\csname the#1\endcsname.}
\makeatother
\usepackage{amsthm}
\usepackage{bm}
\begin{document}
\newtheoremstyle{mytheorem}
  {\topsep}   
  {\topsep}   
  {\itshape}  
  {}       
  {\bfseries} 
  {. }         
  {5pt plus 1pt minus 1pt} 
  { }          
\newtheoremstyle{myremark}
  {\topsep}   
  {\topsep}   
  {\upshape}  
  {}       
  {\itshape} 
  {. }         
  {5pt plus 0pt minus 1pt} 
  {}          
\theoremstyle{mytheorem}
\newtheorem*{A}{Li--Zheng theorem}
\newtheorem*{problem}{Problem}
\newtheorem{theorem}{Theorem}[section]
 \newtheorem{theorema}{Theorem}
\newtheorem{proposition}[theorem]{Proposition} 
 \newtheorem{lemma}[theorem]{Lemma}
\newtheorem{corollary}[theorem]{Corollary}
\newtheorem{definition}[theorem]{Definition}
\theoremstyle{myremark}
\newtheorem{remark}[theorem]{Remark}

\begin{Large}
\noindent{\bf On the Li--Zheng theorem}

\bigskip

\title{On the Li--Zheng theorem}

\bigskip

\noindent{Gennadiy Feldman} 
\end{Large}
\bigskip

\noindent {\bf  Abstract.}  By the well-known
 I.~Kotlarski lemma,
if $\xi_1$, $\xi_2$, and $\xi_3$
are independent real-valued random
variables with nonvanishing characteristic 
functions, $L_1=\xi_1-\xi_3$ and  $L_2=\xi_2-\xi_3$, then 
 the distribution of the random vector  $(L_1, L_2)$ determines 
  the 
distributions of the random variables $\xi_j$ up
 to shift.  Siran Li and  Xunjie Zheng generalized 
 this result for the linear forms $L_1=\xi_1+a_2\xi_2+a_3\xi_3$ and $L_2=b_2\xi_2+b_3\xi_3+\xi_4$ assuming that all   
 $\xi_j$ have 
 first and second moments, $\xi_2$ and $\xi_3$ are identically distributed, 
  and $a_j$, $b_j$ satisfy some conditions.
 In the article, we give a simpler proof of this theorem. 
 In doing so, we also prove that 
the condition of existence of moments 
can be omitted. Moreover, we
prove an analogue of the Li--Zheng theorem for independent random
variables with values in the field 
of $p$-adic numbers, in the field of integers modulo $p$, where $p\ne 2$, and 
in the discrete field of rational numbers.
     
\bigskip


\bigskip


\section { Introduction}

In the article \cite{K} dedicated to the characterization of the gamma and
		the Gaussian distribution I.~Kotlarski proved the following lemma:
Let $\xi_1$, $\xi_2$, and $\xi_3$
 be independent real-valued random
variables with nonvanishing characteristic 
functions,  and let
$L_1=\xi_1-\xi_3$ and  $L_2=\xi_2-\xi_3$. 
Then the distribution of the random vector  $(L_1, L_2)$ determines the 
distributions  of the random variables $\xi_j$ up
 to shift.	The characterization theorems proved in \cite{K} 
 remained known only to mathematicians who deal with characterization problems of mathematical statistics. At the same time, a number of studies that are 
 far from characterization problems are based on Kotlarski's lemma. See, e.g.,
 \cite{L-Zh},  where  numerous articles with references to Kotlarski's 
 lemma including
 in economics literature, are mentioned. We especially pay attention 
 to the important article by C.R.~Rao
 \cite{CRR}. In particular, he considered $n$ independent real-valued 
 random variables $\xi_j$ 
 with nonvanishing characteristic 
functions,
 the linear forms $L_1=a_1\xi_1+\dots+a_n\xi_n$,  $L_2=b_1\xi_1+\dots+b_n\xi_n$
 and proved that under some natural conditions on the coefficients  $a_j$,
 $b_j$ the characteristic  function of the random vector  $(L_1, L_2)$ determines the
  characteristic  functions of the random variables 
  $\xi_j$ up to  factors of the form $\exp\{P_j(y)\}$, 
  where $P_j(y)$ is a polynomial of degree at most $n-2$. Kotlarski's lemma follows
  from this Rao theorem. Note also that some generalizations of Rao's theorem
  were studied  in \cite{F}, see also \cite[\S 15]{F1}, for locally compact Abelian groups. The proof of  Rao's theorem is based on the following statement on solutions of a functional equation.
    \begin{lemma}\label{le1}    
Let $a_j$, $b_j$, $j=1,2, \dots,n$,   be nonzero real numbers such that
  $a_ib_j\ne a_jb_i$  for all $i\ne j$.  Consider the  equation
$$
    \sum_{j = 1}^{n}  \psi_j(a_ju + b_j v ) = A(u)+B(v),
\quad u, v \in \mathbb{R},
$$ 
where $\psi_j(y)$, $A(y)$, and $B(y)$   are continuous complex-valued 
functions on $\mathbb{R}$. 
Then $\psi_j(y)$ are polynomial on $\mathbb{R}$ of degree at most $n$. 
\end{lemma}
This lemma is well known. In fact, the lemma was first used, although it was not explicitly formulated, in the works by Skitovich and Darmois, where the Gaussian distribution on the real line is characterized by the independence of two 
linear forms of $n$ independent random variables. Two different proofs of the lemma 
and some generalizations can be found in \cite[\S 1.5]{KaLiRa}.

In article \cite{L-Zh}, Siran Li and  Xunjie Zheng proved the 
following statement.

\begin{A}
Let U and V be random variables with 
nonvanishing characteristic functions. Assume that
$$U = X + aZ_1 + bZ_2, \quad V = Y + cZ_1 + dZ_2,
$$
where $X$, $Y$, $Z_1$, and $Z_2$ are independent random variables 
with well-defined first and second moments, $Z_1$ and $Z_2$ are identically
distributed, and $a$, $b$, $c$, and $d$ are nonzero real constants 
which are known. Suppose that $ac \ne -bd$ and $(a, c) \ne (-b, -d)$.
Then, the joint distribution of $(U, V)$ uniquely determines 
the distributions of $X$, $Y$, $Z_1$, and $Z_2$ up to a change of
location.
\end{A}

Our article consists of three parts. In the first part,
we give a simpler proof of the Li--Zheng theorem and
 show that it essentially follows 
from Lemma \ref{le1}. In doing so, we also prove that 
the condition  of existence of moments of the random 
variables $X$, $Y$, $Z_1$, and $Z_2$   
can be omitted. In the second part, we
prove an analogue of the Li--Zheng theorem for independent random
variables with values in  the field 
of $p$-adic numbers. 
In the third part, we prove an analogue of the Li--Zheng theorem 
for independent random
variables with values in the field of integers modulo 
$p$, where $p\ne 2$,  and 
in the discrete field of rational numbers.
For the proof of the corresponding theorems we 
  solve some functional equations on the character group
of the additive group of the field.
		
\section { Real-valued random variables} 
 
Let us formulate the Li--Zheng theorem in more familiar for us notation. 
In doing so, we omit
the condition  of  existence of moments of independent random 
variables.

\begin{theorem}\label{th1}  Let $\xi_1$,  $\xi_2$,  $\xi_3$, and 
$\xi_4$ be independent
real-valued random variables with nonvanishing characteristic functions. 
Let $a_j$, $b_j$, $j=2,3$,   be nonzero real numbers such that
  $a_2b_2\ne -a_3b_3$ and $(a_2, b_2) \ne (-a_3, -b_3)$.
Consider the linear forms $L_1=\xi_1+a_2\xi_2+a_3\xi_3$ 
and $L_2=b_2\xi_2+b_3\xi_3+\xi_4$. 
If the random variables $\xi_2$ and  $\xi_3$ are identically 
distributed, then the distribution of 
the random vector $(L_1, L_2)$  determines 
  the 
distributions of  the random variables 
$\xi_j$, $j=1, 2, 3, 4$, up to shift. 
\end{theorem}
\begin{proof} 1. Let $\eta_1$,  $\eta_2$,  $\eta_3$, and $\eta_4$ 
be independent real-valued random variables with nonvanishing 
characteristic functions.
Assume that
$\eta_2$ and  $\eta_3$ are identically distributed.
 Denote by $\mu_j$ and $\nu_j$ the distributions of the random variables   
 $\xi_j$  and $\eta_j$ and by $\hat\mu_j(y)$ and $\hat\nu_j(y)$ their
 characteristic functions.
 Put $M_1=\eta_1+a_2\eta_2+a_3\eta_3$ and  $M_2=b_2\eta_2+b_3\eta_3+\eta_4$. 
Suppose that the distributions of the random vectors   $(L_1, L_2)$ 
 and $(M_1, M_2)$ coincide. Taking into account that the random variables 
 $\xi_j$ are 
independent and $\hat\mu_j(y)=\textbf{E}\left[e^{i\xi_jy}\right]$, 
the characteristic   function of the random 
vector   $(L_1, L_2)$  can be represented in the   form
\begin{multline}\label{30_03_2}
\textbf{E}\left[e^{i(L_1u+L_2v)}\right]=\textbf{E}
\left[e^{i((\xi_1+a_2\xi_2+a_3\xi_3)u+(b_2\xi_2+b_3\xi_3+\xi_4)v)}\right]\\ 
=\textbf{E}\left[e^{i\xi_1u}e^{i\xi_2(a_2u+b_2v)}e^{i\xi_3(a_3u+b_3v)}
e^{i\xi_4v}\right]
=\textbf{E}\left[e^{i\xi_1u}\right]\textbf{E}\left[e^{i\xi_2(a_2u+b_2v)}\right]
\textbf{E}\left[e^{i\xi_3(a_3u+b_3v)}\right]\\ 
\times\textbf{E}\left[e^{i\xi_4v}\right]
=\hat\mu_1(u)
\hat\mu_2(a_2u+b_2v)\hat\mu_3(a_3u+b_3v)\hat\mu_4(v), 
\quad u, v\in \mathbb{R}.
\end{multline}
Analogically, the characteristic   function of the random 
vector   $(M_1, M_2)$  is of the   form
\begin{equation}\label{30_03_3}
\textbf{E}\left[e^{i(M_1u+M_2v)}\right]=\hat\nu_1(u)
\hat\nu_2(a_2u+b_2v)
\hat\nu_3(a_3u+b_3v)\hat\nu_4(v), \quad u, v\in \mathbb{R}.
\end{equation}
It follows from  (\ref{30_03_2})  and (\ref{30_03_3}) that
the random vectors $(L_1, L_2)$ and $(M_1, M_2)$ have the same 
characteristic functions and hence  they are identically distributed 
if and only if  the 
characteristic functions $\hat\mu_j(y)$ and $\hat\nu_j(y)$ 
satisfy the equation
\begin{equation}\label{30_03_4}
\hat\mu_1(u)\hat\mu_2(a_2u+b_2v)
\hat\mu_3(a_3u+b_3v)\hat\mu_4(v)=
\hat\nu_1(u)
\hat\nu_2(a_2u+b_2v)\hat\nu_3(a_3u+b_3v)\hat\nu_4(v), 
\quad u, v\in \mathbb{R}.
\end{equation} 
Set  
  \begin{equation}\label{30_03_1}
f_j(y)=\hat\nu_j(y)/\hat\mu_j(y), \quad \psi_j(y)=\ln f_j(y), 
\quad   j=1, 2, 3, 4.
\end{equation} 
Since the characteristic functions $\hat\mu_j(y)$ and $\hat\nu_j(y)$ 
do not vanish, (\ref{30_03_4}) is equivalent to the fact that 
the functions $f_j(y)$ satisfy the equation
\begin{equation}\label{30_03_5}
f_1(u)f_2(a_2u+b_2v)
f_3(a_3u+b_3v)f_4(v)=1, \quad u, v\in \mathbb{R}.
\end{equation}

Note that we obtained  equation (\ref{30_03_5}) without assuming that
the random variables $\xi_2$ and  $\xi_3$ 
and also $\eta_2$ and  $\eta_3$ are identically distributed.

It follows from  (\ref{30_03_5}) that the functions  $\psi_j(y)$ 
  satisfy the equation
\begin{equation}\label{30_03_6}
\psi_1(u)+\psi_2(a_2u+  b_2v)
+\psi_3(a_3u+  b_3v)+\psi_4(v)=0, 
\quad u, v\in \mathbb{R}.
\end{equation}  
Rewrite this equation in the form 
\begin{equation}\label{01_04_1}
\psi_2(a_2u+  b_2v)
+\psi_3(a_3u+  b_3v)=A(u)+B(v), 
\quad u, v\in \mathbb{R}.
\end{equation}

2. First assume that  $a_2b_3\ne a_3b_2$. Then by Lemma \ref{le1}, 
the functions 
$\psi_2(y)$ and $\psi_3(y)$
are polynomial of degree at most 2. Taking into account 
that $\psi_2(0)=\psi_3(0)=0$, we have
\begin{equation}\label{30_03_7}
\psi_j(y)=\sigma_jy^2+\beta_jy, \quad y\in \mathbb{R}, \ 
j=2, 3,
\end{equation}
where $\sigma_j$, $\beta_j$ are complex numbers. Substituting 
(\ref{30_03_7}) into (\ref{30_03_6}) and  setting first $v=0$ and then
$u=0$ in the obtained equation, we infer 
\begin{equation}\label{02_04_1}
\psi_j(y)=\sigma_jy^2+\beta_jy, \quad y\in \mathbb{R}, \ 
j=1, 4,
\end{equation}
where $\sigma_j$, $\beta_j$ are complex numbers. Substitute
(\ref{30_03_7}) and (\ref{02_04_1}) into (\ref{30_03_6}).
We get from the received equation
\begin{equation}\label{30_03_8}
\sigma_1u^2+\sigma_2(a_2u+b_2v)^2+\sigma_3(a_3u+b_3v)^2+\sigma_4v^2=0, 
\quad u, v\in \mathbb{R}.
\end{equation}
It follows from (\ref{30_03_8}) that 
\begin{equation}\label{30_03_9}
\sigma_2a_2b_2+\sigma_3a_3b_3=0.
\end{equation}
 Since  $\xi_2$ and $\xi_3$ are 
 identically distributed and $\eta_2$ 
 and  $\eta_3$ are also identically distributed, we 
 have $\hat\mu_2(y)=\hat\mu_3(y)$ and 
$\hat\nu_2(y)=\hat\nu_3(y)$, $y\in \mathbb{R}$. Hence 
$\psi_2(y)=\psi_3(y)$, $y\in \mathbb{R}$. This implies in 
particular that 
$\sigma_2=\sigma_3$. Inasmuch as $a_2b_2+a_3b_3\ne 0$ by 
the conditions of the theorem, we get from  
(\ref{30_03_9})  that 
$\sigma_2=\sigma_3=0$. Then it follows from (\ref{30_03_8}) that 
$\sigma_1=\sigma_4=0$. Thus 
\begin{equation}\label{01_04_3}
\psi_j(y)=\beta_jy, \quad y \in \mathbb{R}, \ j=1, 2, 3, 4.
\end{equation}
 In view of (\ref{30_03_1}), 
$\psi_j(-y)=\overline{\psi_j(y)}$ for all $y \in \mathbb{R}$ and
(\ref{01_04_3}) implies that 
$\beta_j=i\alpha_j$, where $\alpha_j$ are real numbers. 
Hence $f_j(y)=e^{i\alpha_j y}$. Thus we proved that
\begin{equation*} 
\hat\nu_j(y)=\hat\mu_j(y)e^{i\alpha_j y}, \quad y \in \mathbb{R}.
\end{equation*}
It follows from this that
\begin{equation*} 
\nu_j=\mu_j*E_{\alpha_j}, \quad j=1, 2, 3, 4,
\end{equation*}
where $E_{\alpha_j}$ is the degenerate distribution concentrated at the point 
$\alpha_j$. So, if $a_2b_3\ne a_3b_2$, the theorem is proved. 

3. Assume now that $a_2b_3=a_3b_2$. In this case, we can not apply 
Lemma \ref{le1} 
for solving equation  (\ref{01_04_1}), but  equation (\ref{01_04_1}) 
can be easily solved directly.  

  Since $\xi_2$ and $\xi_3$ are identically distributed and $\eta_2$ 
  and  $\eta_3$ are also identically distributed, we have 
  $\psi_2(y)=\psi_3(y)$, $y\in\mathbb{R}$. 
  Put \begin{equation}\label{13.1}
  \psi(y)=\psi_2(y)=\psi_3(y).
  \end{equation}
 In view of  $a_2b_3=a_3b_2$, set $c=b_2/a_2=b_3/a_3$. 
 Inasmuch as $a_2u+b_2v=a_2(u+cv)$ and
$a_3u+b_3v=a_3(u+cv)$,
 it is easy to see that equation (\ref{01_04_1}) can be rewritten 
in the form
\begin{equation}\label{30_03_10}
\psi(a_2(u+cv))+\psi(a_3(u+cv))=\psi(a_2u)+\psi(a_3u)+\psi(a_2cv)+\psi(a_3cv), 
\quad u, v\in \mathbb{R}.
\end{equation}

Let $a_2=a_3$. Then (\ref{30_03_10}) implies that $\psi(y)$ is a  
homogeneous linear function. If we consider this fact, it follows from 
(\ref{30_03_6}) and (\ref{13.1}) that  $\psi_1(y)$ and $\psi_4(y)$ are also 
homogeneous linear functions, i.e., (\ref{01_04_3})  is fulfilled. 
As noted in the proof of the final part
of item 2, 
the statement of the theorem follows from this. 
 
Let $a_2\ne a_3$. Put $$\varphi(y)=\psi(a_2y)+\psi(a_3y).$$ 
It follows from (\ref{30_03_10}) 
that the function $\varphi(y)$ satisfies the equation
\begin{equation*} 
\varphi(u+v)=\varphi(u)+\varphi(v), 
\quad u, v\in \mathbb{R}.
\end{equation*}
Hence there is a complex number $a$ such that 
$\varphi(y)=a y$ for all $y\in \mathbb{R}$. Consider the function
\begin{equation}\label{01_04_5}
\gamma(y)=\psi(y)-by, \quad y\in \mathbb{R},
\end{equation}
where $b=\frac{a}{a_2+a_3}$. Then we have
$$
\gamma(a_2y)+\gamma(a_3y)=\psi(a_2y)-ba_2y+\psi(a_3y)-ba_3y=\varphi(y)-a y=0.
$$
This implies that
\begin{equation}\label{nn1}
\gamma(y)=-\gamma(ky), \quad y\in \mathbb{R},
\end{equation}
where $k=\frac{a_2}{a_3}$. 
Since $a_2b_3=a_3b_2$ 
and $(a_2, b_2) \ne (-a_3, -b_3)$, we have $k\ne -1$. 
By the condition, $k\ne 1$. For this reason $|k|\ne 1$.
Suppose for definiteness that $|k|<1$. We get from (\ref{nn1})
\begin{equation}\label{n1}
\gamma(y)=(-1)^n\gamma(k^ny), \quad y\in \mathbb{R}, \ n=1, 2, \dots 
\end{equation}
Obviously, $k^ny\rightarrow 0$ as $n\rightarrow\infty$ for all $y\in \mathbb{R}$. 
Taking  into account that $\gamma(y)$ is a continuous function and 
$\gamma(0)=0$, it follows from (\ref{n1}) that $\gamma(y)=0$ for all
$y\in \mathbb{R}$. Hence (\ref{01_04_5}) implies that $\psi(y)$ is a  
homogeneous linear function. Then it follows from 
(\ref{30_03_6}) that $\psi_1(y)$ and $\psi_4(y)$ are also 
homogeneous linear functions, i.e.,
(\ref{01_04_3}) is fulfilled. As noted above, the statement of the 
theorem follows from this. The theorem is completely proved.
 \end{proof}
 
Let $a_j$, $b_j$, $j=2,3$,   be nonzero real numbers. It is obvious that 
if $a_2b_3\ne a_3b_2$, then $(a_2, b_2) \ne (-a_3, -b_3)$.
Assume that $a_2b_3=a_3b_2$. 
This implies that $a_2b_2\ne -a_3b_3$, and the condition $(a_2, b_2) \ne (-a_3, -b_3)$
is equivalent to the condition that either $|a_2|\ne |a_3|$ or $a_2=a_3$.
Taking this into account, Theorem \ref{th1}
  can be reformulated as follows (compare below with Theorems 
  \ref{th2}, \ref{th4}, and \ref{th5}).
 \begin{theorem}\label{th3}
Let $\xi_1$,  $\xi_2$,  $\xi_3$, and $\xi_4$ be independent
 random variables with values in $\mathbb{R}$ with 
 nonvanishing characteristic functions. 
Suppose that the random variables
$\xi_2$ and  $\xi_3$ are identically  distributed. 
Let $a_j$, $b_j$, $j=2,3$,   be nonzero real numbers. 
Consider the linear forms $L_1=\xi_1+a_2\xi_2+a_3\xi_3$ and 
$L_2=b_2\xi_2+b_3\xi_3+\xi_4$. 
Assume that one of the following conditions holds:
\renewcommand{\labelenumi}{\rm(\Roman{enumi})}
\begin{enumerate} 
\item
 $a_2b_3\ne a_3b_2$ and $a_2b_2\ne -a_3b_3$; 
 \item
$a_2b_3=a_3b_2$ and $|a_2|\ne |a_3|$;
\item
$a_2b_3=a_3b_2$ and $a_2=a_3$.
\end{enumerate} 
Then  the distribution of 
the random vector $(L_1, L_2)$  determines 
  the 
distributions of the random variables 
$\xi_j$, $j=1, 2, 3, 4$, up to shift. 
\end{theorem}

\begin{remark}\label{re1}
Let us assume that in Theorem \ref{th3} the coefficients
$a_j$ and $b_j$ satisfy the condition

\medskip
 
\noindent(IV) $a_2b_3=a_3b_2$ and $a_2=-a_3$.

\medskip

We will verify that in this case the distribution 
of the random vector $(L_1, L_2)$ uniquely determines  the 
distributions of the random variables
$\xi_1$ and $\xi_4$, i.e., $\nu_1=\mu_1$ and $\nu_4=\mu_4$, but 
 need not necessarily 
determines the distribution of the random variables
$\xi_2$ and $\xi_3$ up to  shift. Taking into account that 
 $|a_2|=|a_3|$ if and only if either $a_2=a_3$ or $a_2=-a_3$, 
it means that we can not strengthen 
Theorem \ref{th3} 
and replace conditions (II) and (III) 
in Theorem \ref{th3} by the condition $a_2b_3=a_3b_2$.

Since the random variables $\xi_2$ and  $\xi_3$ are 
identically  distributed and $\eta_2$ and  $\eta_3$ are 
also identically  distributed, put 
$f(y)=f_2(y)=f_3(y)$. In view of (IV), equation (\ref{30_03_5}) takes the form
\begin{equation}\label{10.1}
f_1(u)|f(a_2(u+cv))|^2
f_4(v)=1, \quad u, v\in \mathbb{R}.
\end{equation}
Set $l(y)=|f(a_2y)|^2$. It follows from  (\ref{10.1}) that the 
function $l(y)$ satisfies the equation
$$
l(u+v)=l(u)l(v), \quad u, v\in \mathbb{R}.
$$
Hence $l(y)=e^{\kappa y}$, where $\kappa\in \mathbb{R}$. Inasmuch as $l(-y)=l(y)$, 
we have $\kappa=0$, i.e., $l(y)=1$ for all $y\in \mathbb{R}$. 
Taking this into account, we get from equation (\ref{10.1})  that
$f_1(y)=f_4(y)=1$ for all $y\in \mathbb{R}$. Hence 
$\hat\nu_1(y)=\hat\mu_1(y)$ and $\hat\nu_4(y)=\hat\mu_4(y)$ 
for all $y\in \mathbb{R}$. This implies that
$\nu_1=\mu_1$ and $\nu_4=\mu_4$.

Consider the distributions $\mu$ and $\nu$ with the characteristic functions
$$ 
\hat\mu(y)=\exp\{(e^{iy}-1)\}, \quad 
\hat\nu(y)=\exp\{(e^{-iy}-1)\}, \quad y\in \mathbb{R}.
$$
Then we have
\begin{equation*}  
|\hat\mu(y)|=|\hat\nu(y)|=\exp\{\cos y-1\}, \quad y\in \mathbb{R}.
\end{equation*}
This implies that
\begin{equation} \label{10.2} 
|f(y)|=1, \quad y\in \mathbb{R}.
\end{equation}
It is obvious that 
 $\nu$ 
is not a shift of $\mu$. Moreover, It is easy to see that there is no 
a distribution $\lambda$ such that either $\nu=\mu*\lambda$ or $\mu=\nu*\lambda$.

Let $\xi_1$,
 $\xi_2$,  $\xi_3$, and $\xi_4$  
be independent random variables such that
$\xi_2$ and  $\xi_3$ are identically distributed. 
Assume that the random variable  
$\xi_j$  
   has the
distribution  $\mu_j$, $j=1, 2, 3, 4$, where $\mu_1$ and
$\mu_4$ are arbitrary distributions  
with nonvanishing characteristic functions and $\mu_2=\mu_3=\mu$.
Consider the linear forms $L_1=\xi_1+a_2\xi_2-a_2\xi_3$ and 
$L_2=b_2\xi_2-b_2\xi_3+\xi_4$. 
Let $\eta_1$,
 $\eta_2$,  $\eta_3$, and $\eta_4$  
be independent random variables   such that
$\eta_2$ and  $\eta_3$ are identically distributed. 
Suppose that $\eta_j$  
   has the
distribution  $\nu_j$, $j=1, 2, 3, 4$, where $\nu_1=\mu_1$, 
$\nu_2=\nu_3=\nu$,  and $\nu_4=\mu_4$.   
Put $M_1=\eta_1+a_2\eta_2-a_2\eta_3$ and 
$M_2=b_2\eta_2-b_2\eta_3+\eta_4$. 

Taking into account that
$f_1(y)=f_4(y)=1$ for all $y\in \mathbb{R}$ and (\ref{10.2}),
we see that the functions $f_1(y)$, $f(y)$, and $f_4(y)$ satisfy 
 equation (\ref{10.1}). Hence the 
 random vectors $(L_1, L_2)$ and $(M_1, M_2)$ are identically distributed, while
 $\nu$ 
is not a shift of $\mu$. 
\end{remark}
 
\section { Random
variables with values in the field 
of $p$-adic numbers}

Let $X$ be a locally compact Abelian group, $Y$ be its character group. Denote by
$(x, y)$ the value of a character $y\in Y$ at an element $x\in X$. Let $\mu$
be a distribution on $X$. Denote by
\begin{equation}\label{e1}
\hat\mu(y)=\int_X (x, y)d\mu(x), \quad y\in Y,
\end{equation}
the characteristic function of the distribution $\mu$.

Let $f(y)$ be a function on  $Y$  and let $h$ be an
element of
$Y$. Denote by $\Delta_h$ the  finite difference operator 
$$\Delta_h f(y)=f(y+h)-f(y), \quad y\in Y.$$ A function $f(y)$ on $Y$ is called
a  polynomial  if
$$\Delta_{h}^{n+1}f(y)=0$$ for some $n$ and for all $y, h \in Y$.

We need the following well-known statement 
(for the proof see, e.g., \cite[Proposition 1.30]{F1}).
\begin{lemma}\label{le2} Let $Y$ be a locally compact Abelian group such that
all its elements are compact. Then any continuous polynomial on $Y$
is a constant.
\end{lemma} 

Consider the field of $p$-adic numbers $\mathbb{Q}_p$. When we say 
the group $\mathbb{Q}_p$, we mean the additive 
group of the field $\mathbb{Q}_p$.
The group $\mathbb{Q}_p$ is a locally compact Abelian group. 
Its character group is
topologically isomorphic to  $\mathbb{Q}_p$ (\!\!\cite[(25.1)]{HeRo1}). 
Multiplication by a nonzero element of $\mathbb{Q}_p$ is a topological 
automorphism of the group $\mathbb{Q}_p$. Note that $(ax, y)=(x, ay)$ 
for all $a, x, y\in \mathbb{Q}_p$. If $\mu$ is a distribution
on $\mathbb{Q}_p$, the characteristic function $\hat\mu(y)$ is defined by
formula (\ref{e1}), where $X=Y=\mathbb{Q}_p$. 
The group $\mathbb{Q}_p$ is totally disconnected and 
consists of compact elements. Denote by $|\cdot|_p$ 
the norm in the field $\mathbb{Q}_p$. 

In this section, we prove the following 
analogue of  the Li--Zheng theorem for the field $\mathbb{Q}_p$.

\begin{theorem}\label{th2}  Let $\xi_1$,  $\xi_2$,  $\xi_3$, and 
$\xi_4$ be independent
random variables with  values in the field $\mathbb{Q}_p$ with
nonvanishing characteristic functions. Let $a_j$, $b_j$, $j=2,3$,   
be nonzero 
elements of $\mathbb{Q}_p$. Consider the linear forms 
$L_1=\xi_1+a_2\xi_2+a_3\xi_3$ and $L_2=b_2\xi_2+b_3\xi_3+\xi_4$. 
Assume that one of the following conditions holds:
\renewcommand{\labelenumi}{\rm(\Roman{enumi})}
\begin{enumerate} 
\item
 $a_2b_3\ne a_3b_2$;
\item
$a_2b_3=a_3b_2$, $|a_2|_p \ne |a_3|_p$, and   the random variables
$\xi_2$ and  $\xi_3$ are identically  distributed. 
\end{enumerate} 
Then  the distribution of  
the random vector $(L_1, L_2)$  determines 
  the 
distributions of the random variables
$\xi_j$, $j=1, 2, 3, 4$, up to shift.
\end{theorem}
\begin{proof} 
Consider independent random variables 
$\eta_1$,  $\eta_2$,  $\eta_3$, and $\eta_4$ with  values in 
the field $\mathbb{Q}_p$ with
nonvanishing characteristic functions. 
Put $M_1=\eta_1+a_2\eta_2+a_3\eta_3$ and  
$M_2=b_2\eta_2+b_3\eta_3+\eta_4$. 
Suppose that the distributions of the random vectors   $(L_1, L_2)$ 
 and $(M_1, M_2)$ coincide.
Note that 
if $\xi$ is a random variable with values
in $\mathbb{Q}_p$ and distribution $\mu$, then 
$\hat\mu(y)=\textbf{E}[(\xi, y)]$.
Taking this into account and the fact that 
$(ax, y)=(x, ay)$ 
for all $a, x, y\in \mathbb{Q}_p$, we can argue as in item 1 of 
the proof of Theorem \ref{th1}.  
Keeping the same notation, we arrive at the equation 
\begin{equation}\label{11.1}
f_1(u)f_2(a_2u+b_2v)
f_3(a_3u+b_3v)f_4(v)=1, \quad u, v\in \mathbb{Q}_p.
\end{equation}

1. Assume that condition (I) holds. Since  
we do not suppose that $\xi_2$ and  $\xi_3$ are identically 
distributed and $\eta_2$ and  $\eta_3$ are identically 
distributed, we can assume, without loss of generality, that 
$L_1=\xi_1+\xi_2+\xi_3$ and  $M_1=\eta_1+\eta_2+\eta_3$.
Then equation (\ref{11.1})  takes the form 
\begin{equation}\label{01_04_6}
f_1(u)f_2(u+b_2v)
f_3(u+b_3v)f_4(v)=1, \quad u, v\in \mathbb{Q}_p,
\end{equation}
and the condition $a_2b_3\ne a_3b_2$ is transform to the condition
$b_2 \ne  b_3$.
The group
$\mathbb{Q}_p$ is totally disconnected. For this reason,   
as opposed to the case of the real line,  we can not take the 
logarithm of both sides of  equation (\ref{01_04_6})
and pass to the corresponding additive equation.

To solve  equation  (\ref{01_04_6}) we use a slightly different approach
and split the solution of equation (\ref{01_04_6}) into two parts. 
First we prove that $|f_j(y)|=1$ for all $y\in \mathbb{Q}_p$, $j=2, 3$. 
Then we prove that the functions $f_j(y)$, $j=1, 2, 3, 4$, are 
characters of the 
$\mathbb{Q}_p$. It is obvious that the statement of the theorem follows from this.

Put $$\theta_j(y)=\ln|f_j(y)|, \quad  j=1, 2, 3, 4.$$  
It follows from (\ref{01_04_6})
that the functions $\theta_j(y)$ satisfy the equation
$$
\theta_1(u)+\theta_2(u+b_2v)
+\theta_3(u+b_3v)+\theta_4(v)=0, 
\quad u, v\in \mathbb{Q}_p,
$$
which can be written in the form
\begin{equation}\label{01_04_7}
\theta_2(u+b_2v)
+\theta_3(u+b_3v)=C(u)+D(v), 
\quad u, v\in \mathbb{Q}_p.
\end{equation}
 For solving   
equation (\ref{01_04_7}) we use the finite difference method. 
The reasoning is standard and the same as in the case of the real line. 
We present it here for completeness.

Let $g$ be an arbitrary element of  the group
$\mathbb{Q}_p$.  Substitute $u-{b_3}g$ for $u$ and
$v+g$ for $v$  in equation (\ref{01_04_7}). Subtracting 
(\ref{01_04_7})  from the   resulting equation we get   
\begin{equation}
\label{01_04_8}
\Delta_{(b_2-b_3)g}{\theta_2(u + b_2 v)}
    =\Delta_{-{b_3}g}C(u)+\Delta_{g}D(v),
\quad u,v\in \mathbb{Q}_p.
\end{equation}
Let $h$ be an arbitrary element of the group
$\mathbb{Q}_p$. Substitute $u+h$ for $u$  in equation 
(\ref{01_04_8}). Subtracting  
(\ref{01_04_8})  from the   resulting equation we obtain 
\begin{equation}
\label{01_04_9}
\Delta_h\Delta_{(b_2-b_3)g}{\theta_2(u + b_2 v)}
    =\Delta_h\Delta_{-{b_3}g}C(u),
\quad u,v\in \mathbb{Q}_p.
\end{equation}
Let $k$ be an arbitrary element of the group
$\mathbb{Q}_p$. Substitute $v+k$ for $v$  in equation 
(\ref{01_04_9}). Subtracting  
(\ref{01_04_9})  from the   resulting equation we get 
\begin{equation}\label{new1}
\Delta_{b_2k}\Delta_h\Delta_{(b_2-b_3)g}
{\theta_2(u + b_2 v)}=0,
\quad u,v\in \mathbb{Q}_p.
\end{equation}
Substituting $v=0$ in equation (\ref{new1}), we obtain
\begin{equation*} 
\Delta_{b_2k}\Delta_h\Delta_{(b_2-b_3)g}
{\theta_2(u)}=0,
\quad u\in \mathbb{Q}_p.
\end{equation*}
Since $b_2-b_3\ne 0$ and $g$, $h$ and $k$ are arbitrary elements 
of the group
$\mathbb{Q}_p$, we conclude that the function  $\theta_2(y)$ satisfies 
the equation
\begin{equation*} 
\Delta^3_{h}{\theta_2(y)}=0,
\quad y,h\in \mathbb{Q}_p,
\end{equation*}
i.e., is a polynomial on $\mathbb{Q}_p$. Since
the group $\mathbb{Q}_p$ consists of compact elements and 
the polynomial $\theta_2(y)$ is continuous, by Lemma \ref{le2},
$\theta_2(y)$ is a constant.
In view of $\theta_2(0)=0$, we have $\theta_2(y)=0$
for all $y\in \mathbb{Q}_p$. For the function $\theta_3(y)$ 
we argue similarly
excluding first the function $\theta_2(y)$ from equation
(\ref{01_04_7}). Thus  we proved that 
$\theta_2(y)=\theta_3(y)=0$ and hence
$$|f_2(y)|=|f_3(y)|=1, \quad y\in \mathbb{Q}_p.$$

Let us prove that the functions $f_j(y)$, $j=1, 2, 3, 4$, 
are characters of the 
group $\mathbb{Q}_p$.
Rewrite equation (\ref{01_04_6}) in the form
\begin{equation}\label{01_04_12}
f_2(u+b_2v)
f_3(u+b_3v)=S(u)T(v), \quad u, v\in \mathbb{Q}_p.
\end{equation}
To solve equation (\ref{01_04_12}), we apply the method 
that was used to prove Theorem 3.1 in \cite{F}, see also  
\cite[Theorem 15.8]{F1}.

Let $g$ be an arbitrary element of the group
$\mathbb{Q}_p$.   Substitute $u-b_3g$ for $u$ and
$v+g$ for $v$  in equation (\ref{01_04_12}). 
Dividing the   resulting equation by equation (\ref{01_04_12}), we get 
\begin{equation}\label{01_04_13}
\frac{f_2(u+b_2v-b_3g+b_2g)}{f_2(u+b_2v)}
=\frac{S(u-b_3g)T(v+g)}{S(u)T(v)}, \quad u, v\in \mathbb{Q}_p.
\end{equation}
Let $h$ be an arbitrary element of the group
$\mathbb{Q}_p$. Substitute $u+h$ for $u$  in equation 
(\ref{01_04_13}). Dividing the   resulting equation 
by equation (\ref{01_04_13}), 
we receive 
\begin{equation}\label{01_04_14}
\frac{f_2(u+b_2v-b_3g+b_2g+h)f_2(u+b_2v)}
{f_2(u+b_2v+h)f_2(u+b_2v-b_3g+b_2g)}
=\frac{S(u-b_3g+h)S(u)}{S(u+h)S(u-b_3g)}, \quad u, v\in \mathbb{Q}_p.
\end{equation} 
Let $k$ be an arbitrary element of the group
$\mathbb{Q}_p$. Substitute $v+k$ for $v$  in equation 
(\ref{01_04_14}). Dividing the   resulting equation by equation 
(\ref{01_04_14}), we obtain 
\begin{multline}\label{01_04_15}
\frac{f_2(u+b_2v-b_3g+b_2g+h+b_2k)f_2(u+b_2v+b_2k)}
{f_2(u+b_2v+h+b_2k)f_2(u+b_2v-b_3g+b_2g+b_2k)}\\ 
\times\frac{f_2(u+b_2v+h)f_2(u+b_2v-b_3g+b_2g)}
{f_2(u+b_2v-b_3g+b_2g+h)
f_2(u+b_2v)}=1, \quad u, v\in \mathbb{Q}_p.
\end{multline}
Substitute in (\ref{01_04_15}) $u=v=0$ and $g=-\frac{b_2k}{b_2-b_3}$. Then we get
\begin{equation}\label{01_04_16}
\frac{f^2_2(h)f_2(b_2k)f_2(-b_2k)}
{f_2(h+b_2k)f_2(h-b_2k)}=1, \quad h, k\in \mathbb{Q}_p.
\end{equation}
Taking into account that 
\begin{equation}\label{01_04_20}
f_2(-y)=\overline{f_2(y)}, \quad  |f_2(y)|=1, \quad y\in \mathbb{Q}_p,
\end{equation} 
we have $f_2(b_2k)f_2(-b_2k)=1$ for all $k\in \mathbb{Q}_p$.
Then it follows from (\ref{01_04_16}) that 
the function $f_2(y)$ 
satisfies the equation
\begin{equation}\label{01_04_17}
f^2_2(u)= f_2(u+v)f_2(u-v), \quad u, v\in \mathbb{Q}_p.
\end{equation}
In view of (\ref{01_04_20}), we get from (\ref{01_04_17})  that
\begin{equation}\label{01_04_18}
f^2_2(u+v)= f^2_2(u)f^2_2(v), \quad u, v\in \mathbb{Q}_p,
\end{equation}
i.e., the function $f^2_2(y)$
is a character of the group $\mathbb{Q}_p$. 
Substituting  $u=v=y$ in (\ref{01_04_17}), we obtain 
\begin{equation}\label{01_04_19}
f^2_2(y)= f_2(2y), \quad y\in \mathbb{Q}_p.
\end{equation}
Taking into account that the mapping $y\rightarrow 2y$ is a topological 
automorphism of the group  $\mathbb{Q}_p$, it follows from  
(\ref{01_04_18}) and (\ref{01_04_19}) that the 
function $f_2(y)$
satisfies the equation
$$
f_2(u+v)= f_2(u)f_2(v), \quad u, v\in \mathbb{Q}_p,
$$
i.e., is a character of the group $\mathbb{Q}_p$. For the function 
$f_3(y)$ we argue similarly. By the Pontryagin duality theorem,
 there are elements
$x_2, x_3\in \mathbb{Q}_p$ such that 
\begin{equation}\label{02_04_2}
f_j(y)=(x_j, y), \quad y\in \mathbb{Q}_p, \ j=2, 3.
\end{equation}
Substituting 
(\ref{02_04_2}) into (\ref{01_04_6}) and  putting first $v=0$ and then
$u=0$ in the obtained equation, we get that the functions $f_1(y)$ and
$f_4(y)$
are also characters of the group $\mathbb{Q}_p$. 

Thus, we have proved that the characteristic functions $\hat\mu_j(y)$, 
$j=1, 2, 3, 4$, are 
determined up to multiplication by a character. Hence
the  distributions $\mu_j$, $j=1, 2, 3, 4$, are 
 determined up to  shift. The theorem is proved if condition (I) is satisfied.

2. Assume that condition (II) holds. 
Put   
\begin{equation}\label{n2}
c=b_2/a_2=b_3/a_3, \quad g(y)=f_2(a_2y)f_3(a_3y),  \quad \tau(y)=\ln|g(y)|.
\end{equation}

2$a$.  We  prove in this item that the functions $f_1(y)$, $f_4(y)$, and $g(y)$ 
are characters of the group $\mathbb{Q}_p$. In doing so,  
we do not assume that $|a_2|_p \ne |a_3|_p$ 
and the random variables 
$\xi_2$ and $\xi_3$ are  
 identically distributed.
 
Since $a_2u+b_2v=a_2(u+cv)$ and $a_3u+b_3v=a_3(u+cv)$,
rewrite equation (\ref{11.1}) in the form
\begin{equation}\label{08_04_10}
f_1(u)g(u+cv)f_4(v)=1, \quad u, v\in \mathbb{Q}_p.
\end{equation}
Substituting first $v=0$ and then
$u=0$ in   (\ref{08_04_10}), we get that 
the function $g(y)$ satisfies the equation
\begin{equation}\label{9.1}
g(u+v)=g(u)g(v), 
\quad u, v\in \mathbb{Q}_p.
\end{equation}
It follows from  
(\ref{9.1})  that the function
 $\tau(y)$ satisfies the equation
\begin{equation*} 
\tau(u+v)=\tau(u)+\tau(v), 
\quad u, v\in \mathbb{Q}_p.
\end{equation*}
Hence $\tau(y)$ is a continuous polynomial. Since
the group $\mathbb{Q}_p$ consists of compact elements, by Lemma \ref{le2},
$\tau(y)$ is a constant.
In view of $\tau(0)=0$, we have $\tau(y)=0$ and hence $|g(y)|=1$ 
for all $y\in \mathbb{Q}_p$. Taking into account (\ref{9.1}), 
this means that the function $g(y)$
 is a character of the group $\mathbb{Q}_p$. By the Pontryagin 
duality theorem,
 there is an element 
$a\in \mathbb{Q}_p$ such that 
\begin{equation}\label{02_04_10}
g(y)=(a, y), \quad y\in \mathbb{Q}_p.
\end{equation}
Hence 
\begin{equation}\label{08_04_11}
g(u+cv)=(a, u+cv), \quad u, v\in \mathbb{Q}_p.
\end{equation}
Substituting 
(\ref{08_04_11}) into (\ref{08_04_10}), we find from the resulting 
equation 
that 
the functions $f_1(y)$ and
$f_4(y)$
are also characters of the group $\mathbb{Q}_p$. 

Note that in the proof 
we did not use the fact that
$|a_2|_p \ne |a_3|_p$ and the random variables $\xi_2$ and $\xi_3$ are  
 identically distributed.

2$b$. Suppose that the random variables $\eta_2$ 
 and  $\eta_3$ are identically distributed. 
 Since the random variables $\xi_2$ and $\xi_3$ are  also 
 identically distributed, we 
 have  $f_2(y)=f_3(y)$.
Set $$f(y)=f_2(y)=f_3(y).$$ 
Then
\begin{equation}\label{11.04.1}
g(y)=f(a_2y)f(a_3y).
\end{equation}
We  will prove now that the function  $f(y)$ is a 
character of the group $\mathbb{Q}_p$.  
Set  
$$
\theta(y)=\ln|f(y)|.$$  
We have 
$$
\tau(y)=\ln|g(y)|=\ln|f(a_2y)f(a_3y)|=\theta(a_2y)+\theta(a_3y), 
\quad y\in \mathbb{Q}_p.
$$ 
Since $\tau(y)=0$ for all $y\in \mathbb{Q}_p$,  it follows from this that 
\begin{equation}\label{9.2}
\theta(a_2y)+\theta(a_3y)=0, \quad y\in \mathbb{Q}_p.
\end{equation}

Inasmuch as
$|a_2|_p\ne|a_3|_p$, assume for definiteness that $|a_2|_p<|a_3|_p$. 
Put $k=\frac{a_2}{a_3}$. Then  $|k|_p<1$. We obtain from (\ref{9.2})
that $\theta(y)=-\theta(ky)$, and hence
\begin{equation} \label{9.3}
\theta(y)=(-1)^n\theta(k^ny), 
\quad y\in \mathbb{Q}_p, \ n=1, 2, \dots.
\end{equation}
It is obvious that  $|k^ny|_p\rightarrow 0$ as $n\rightarrow\infty$ for all 
$y\in \mathbb{Q}_p$. Since $\theta(y)$
 is a continuous function  and 
$\theta(0)=0$, we obtain from (\ref{9.3}) that 
$\theta(y)=0$ and hence $|f(y)|=1$ for all 
$y\in \mathbb{Q}_p$.  
 
Since $|a_2|_p\ne|a_3|_p$, we have $a_2+a_3\ne 0$. Put 
\begin{equation}\label{02_04_9}
b=\frac{a}{a_2+a_3}, \quad h(y)=f(y)(-b, y).
\end{equation}  
 It follows from   
(\ref{02_04_10}), (\ref{11.04.1}), and (\ref{02_04_9}) that
\begin{multline*}
h(a_2y)h(a_3y)=f(a_2y)(-b, a_2y)f(a_3y)(-b, a_3y)\\=g(y)(-b(a_2+a_3), y)
=g(y)(-a, y)=1, \quad  y\in \mathbb{Q}_p.
\end{multline*}
Hence $h(y)=h^{-1}(ky)$.  This implies that
\begin{equation}\label{02_04_11}
h(y)=h^{(-1)^n}(k^ny), 
\quad y\in \mathbb{Q}_p, \ n=1, 2, \dots.
\end{equation}
We have $|k^ny|_p\rightarrow 0$ as $n\rightarrow\infty$ for all 
$y\in \mathbb{Q}_p$. Inasmuch as $h(y)$
 is a continuous function  and 
$h(0)=1$, it follows from (\ref{02_04_11}) that 
$h(y)=1$ for all $y\in \mathbb{Q}_p$, and (\ref{02_04_9}) implies that 
\begin{equation}\label{02_04_12}
f(y)=(b, y), 
\quad y\in \mathbb{Q}_p.
\end{equation}
Taking into account (\ref{02_04_12}) and the fact that $f_1(y)$ and
$f_4(y)$
are also characters of the group $\mathbb{Q}_p$, we see
 that the characteristic 
functions $\hat\mu_j(y)$ are 
determined up to multiplication by a character. Hence
the  distributions $\mu_j$ are 
 determined up to  shift.
 The theorem is proved if condition (II) is satisfied, and hence
is completely proved.
\end{proof} 
\begin{corollary}\label{co1}
Let $\xi_1$,  $\xi_2$,  $\xi_3$, and 
$\xi_4$ be independent
random variables with  values in the field $\mathbb{Q}_p$ with
nonvanishing characteristic functions. Let $a_j$, $b_j$, $j=2,3$,   
be nonzero 
elements of $\mathbb{Q}_p$. Consider the linear forms 
$L_1=\xi_1+a_2\xi_2+a_3\xi_3$ and $L_2=b_2\xi_2+b_3\xi_3+\xi_4$. 
Then  the distribution of  
the random vector $(L_1, L_2)$  determines 
  the 
distributions of the random variables
$\xi_1$ and $\xi_4$ up to shift.
\end{corollary}
 
\begin{proof}
By Theorem \ref{th2}, in the case, when $a_2b_3\ne a_3b_2$, the statement 
of the corollary is true. If $a_2b_3=a_3b_2$, as has been proved in item 
2$a$ of the proof of Theorem \ref{th2} the functions $f_1(y)$ and  $f_4(y)$  
are characters of the group $\mathbb{Q}_p$. The statement 
of the corollary follows from this.
\end{proof}

We can not omit the condition $|a_2|_p\ne|a_3|_p$ in condition 
(II) of Theorem \ref{th2}. 
Namely, the following statement is true.
\begin{proposition}\label{pr1}
Let   $a_j$, $b_j$, $j=2,3$,   
be nonzero 
elements of the field $\mathbb{Q}_p$  such that $a_2b_3=a_3b_2$. 
Assume that
$|a_2|_p=|a_3|_p$. Then there are independent random variables 
$\xi_1$, $\xi_2$,  $\xi_3$, and $\xi_4$  
  with values in $\mathbb{Q}_p$ with nonvanishing 
  characteristic functions such that the following is true:
\renewcommand{\labelenumi}{\rm(\Roman{enumi})}
\begin{enumerate} 
\item
the random variables $\xi_2$ and  $\xi_3$ are identically distributed; 
 \item
the distribution 
of the random vector $(L_1, L_2)$, where $L_1=\xi_1+a_2\xi_2+a_3\xi_3$ and $L_2=b_2\xi_2+b_3\xi_3+\xi_4$,  need not necessarily 
determine the distribution of the random variables
$\xi_2$ and $\xi_3$ up to  shift. 
\end{enumerate}   
\end{proposition}

\begin{proof}
Let $G$ be a second countable compact Abelian group, $H$ be 
its character group. 
 Then $H$ is a countable discrete Abelian group. 
  Let $\alpha(h)$ be a real-valued nonvanishing function
on the group $H$ satisfying the conditions:
\renewcommand{\labelenumi}{\rm(\roman{enumi})}
\begin{enumerate} 
\item
 $\alpha(0)=1$;
\item
$\alpha(-h)= \alpha(h)$ for all $h\in H$; 
\item
$\sum\limits_{h\in H}|\alpha(h)|<2$; 
\end{enumerate}
Consider on the group $G$ the function
$$
\rho(g)=\sum\limits_{h\in H}\alpha(h)\overline{(g, h)}, \quad g\in G.
$$
It follows from (i)--(iii) that $\rho(g)$ 
is the nonnegative density with respect to 
the Haar distribution on $G$ of  a distribution $\mu$
on the group $G$ with the characteristic function $\hat\mu(h)=\alpha(h)$.

Put
$$
\beta(h)=
\begin{cases}
1& \text{\ if\ }\ \ h=0,
\\-\alpha(h)& \text{\ if\ }\ \ h\ne 0.
\end{cases}
$$
Then $\beta(h)$ is a real-valued nonvanishing function on the group $H$. 
Obviously,  the function $\beta(h)$ also
satisfies conditions (i)--(iii).  Hence there is
 a distribution $\nu$
on the group $G$ with the characteristic function $\hat\nu(h)=\beta(h)$.
It is easy to see that 
if $G$ is not isomorphic to the additive group of the 
integers modulo 2, then $\nu$ 
is not a shift of $\mu$.  

Let $G=\mathbb{Z}_p$ be the ring of $p$-adic integers. 
Then 
$\mathbb{Z}_p$ is a compact subgroup of the group $\mathbb{Q}_p$. 
Let $\mu$ and $\nu$ be
the distributions on $\mathbb{Z}_p$ constructed above.
Let $\xi_1$,
 $\xi_2$,  $\xi_3$, and $\xi_4$  
be independent random variables with values in $\mathbb{Z}_p$ such that
$\xi_2$ and  $\xi_3$ are identically distributed. 
Assume that the random variable  
$\xi_j$  
   has the
distribution  $\mu_j$, $j=1, 2, 3, 4$, where $\mu_1$ and
$\mu_4$ are arbitrary distributions on $\mathbb{Z}_p$ with 
  nonvanishing characteristic functions and $\mu_2=\mu_3=\mu$.
Consider the linear forms $L_1=\xi_1+a_2\xi_2+a_3\xi_3$ and 
$L_2=b_2\xi_2+b_3\xi_3+\xi_4$. 
Let $\eta_1$,
 $\eta_2$,  $\eta_3$, and $\eta_4$  
be independent random variables with values in $\mathbb{Z}_p$ such that
$\eta_2$ and  $\eta_3$ are identically distributed. 
Suppose that $\eta_j$  
   has the
distribution  $\nu_j$, $j=1, 2, 3, 4$, where $\nu_1=\mu_1$, 
$\nu_2=\nu_3=\nu$,  and $\nu_4=\mu_4$.   
Put $M_1=\eta_1+a_2\eta_2+a_2\eta_3$ and 
$M_2=b_2\eta_2+b_2\eta_3+\eta_4$. 

Consider  $\xi_j$ and $\eta_j$ as 
independent random variables with values in the field $\mathbb{Q}_p$
and
 verify that the 
characteristic functions of the random vectors $(L_1, L_2)$ 
and $(M_1, M_2)$
coincide. This is equivalent to the fact that the functions 
$f_j(y)=\hat\nu_j(y)/\hat\mu_j(y)$ satisfy  equation 
(\ref{11.1}), where $f_2(y)=f_3(y)=f(y)$. 
Taking into account that
$f_1(y)=f_4(y)=1$ for all $y\in \mathbb{Q}_p$, equation (\ref{11.1})
takes the form
\begin{equation*} 
f(a_2(u+cv))
f(a_3(u+cv))=1, \quad u, v\in \mathbb{Q}_p,
\end{equation*}
where $c=b_2/a_2=b_3/a_3$, or
\begin{equation} \label{08_04_15}
f(a_2y)
f(a_3y)=1, \quad y\in \mathbb{Q}_p.
\end{equation}
We consider the distributions $\mu$ and $\nu$ as distributions on
$\mathbb{Q}_p$. It is easy to see that then the function $f(y)$ is 
of the form
\begin{equation}\label{9.4}
f(y)=
\begin{cases}
1& \text{\ if\ }\ \ y\in p\mathbb{Z}_p,
\\-1& \text{\ if\ }\ \ y\notin p\mathbb{Z}_p.
\end{cases}
\end{equation} 
It  follows from (\ref{9.4}) and the condition 
$|a_2|_p=|a_3|_p$ that the function $f(y)$ satisfies equation (\ref{08_04_15}). Thus,  the 
characteristic functions of the random vectors $(L_1, L_2)$ 
and $(M_1, M_2)$
coincide. Hence  the random vectors $(L_1, L_2)$ 
and $(M_1, M_2)$ are identically distributed, while $\nu$ 
is not a shift of $\mu$.  

We note that if $|a_2|_p\ne|a_3|_p$, then the function $f(y)$ 
does not satisfy equation (\ref{08_04_15}), i.e., 
the random vectors $(L_1, L_2)$ 
and $(M_1, M_2)$ are not identically distributed.
\end{proof}

The condition $|a_2|_p=|a_3|_p$ is an analogue for the field
$\mathbb{Q}_p$ of the condition $|a_2| =|a_3| $ for the field
$\mathbb{R}$. Comparing  the statement of Theorem \ref{th3} in the case
 when condition
(III) holds  and  Proposition \ref{pr1}, we see that in the field
$\mathbb{Q}_p$
do not exist nonzero elements $a_j$, $b_j$, $j=2,3$,   
 such that $a_2b_3=a_3b_2$ and $|a_2|_p=|a_3|_p$ and the distribution of 
the random vector $(L_1, L_2)$  determines 
  the 
distributions of the random variables 
$\xi_2$ and   $\xi_3$ up to shift.

\section { Random
variables with values in the field of integers modulo $p$, where $p\ne 2$,  and 
in the discrete field of rational numbers}

Let $p$ be a prime number and  $\mathbb{Z}(p)$ be the field of 
integers modulo $p$. When we say the group $\mathbb{Z}(p)$, we 
mean the additive group of the field $\mathbb{Z}(p)$, i.e., 
the additive group of the integers modulo $p$. 
The character group of the group $\mathbb{Z}(p)$ is
isomorphic to  $\mathbb{Z}(p)$. 
Multiplication by a nonzero element of $\mathbb{Z}(p)$ is an  
automorphism of the group $\mathbb{Z}(p)$. Note that $(ax, y)=(x, ay)$ 
for all $a, x, y\in \mathbb{Z}(p)$. If $\mu$ is a distribution
on $\mathbb{Z}(p)$, the characteristic function $\hat\mu(y)$ is defined by
formula (\ref{e1}), where $X=Y=\mathbb{Z}(p)$. The group $\mathbb{Z}(p)$ 
is finite and hence compact. 

In this section, we prove  
analogues of  the Li--Zheng theorem for the field of integers 
modulo $p$, where $p\ne 2$, and for the discrete field of 
rational numbers $\mathbb{Q}$.

\begin{theorem}\label{th4}  Consider the field $\mathbb{Z}(p)$, where $p\ne 2$.
  Let $\xi_1$,  $\xi_2$,  $\xi_3$, and 
$\xi_4$ be independent
random variables with  values in $\mathbb{Z}(p)$  with
nonvanishing characteristic functions. Let $a_j$, $b_j$, $j=2,3$,   
be nonzero 
elements of $\mathbb{Z}(p)$. Consider the linear forms 
$L_1=\xi_1+a_2\xi_2+a_3\xi_3$ and $L_2=b_2\xi_2+b_3\xi_3+\xi_4$. 
Then the following statements hold.
\renewcommand{\labelenumi}{\rm\arabic{enumi}.}
\begin{enumerate} 
\item
If $a_2b_3\ne a_3b_2$, then the distribution of  
the random vector $(L_1, L_2)$  determines 
  the 
distributions of the random variables
$\xi_j$, $j=1, 2, 3, 4$, up to shift.
\item
If $a_2b_3=a_3b_2$, then the distribution of  
the random vector $(L_1, L_2)$  determines 
  the 
distributions of the random variables
$\xi_1$  and $\xi_4$ up to shift.
\end{enumerate}
\end{theorem}
\begin{proof}
1. Assume that $a_2b_3\ne a_3b_2$. The proof of   Theorem
 \ref{th2} in the case, when condition (I) holds, is based only on the following properties of the group $\mathbb{Q}_p$: the group $\mathbb{Q}_p$ consists of compact elements and the mapping $y\rightarrow 2y$ is a topological 
automorphism of the group  $\mathbb{Q}_p$. Both of these properties are 
also valid for the group $\mathbb{Z}(p)$, where $p\ne 2$. 
For this reason  the proof  
 remains unchanged if, instead of the field  $\mathbb{Q}_p$, 
 we consider the filed $\mathbb{Z}(p)$. Thus, statement 1 is valid.

2. Assume that $a_2b_3=a_3b_2$. The reasoning carried out 
in item 2$a$ of the proof of Theorem \ref{th2} is based only on the fact that
the group $\mathbb{Q}_p$ consists of compact elements. 
For this reason the proof remains unchanged if, instead of the field  
$\mathbb{Q}_p$,  we consider the filed $\mathbb{Z}(p)$. Thus, statement 
2 is valid.
\end{proof}

Theorem \ref{th4} can not be strengthened. Namely, the following statement is true.
\begin{proposition}\label{pr2}
Let   $a_j$, $b_j$, $j=2,3$,   
be nonzero 
elements of the field $\mathbb{Z}(p)$, where $p\ne 2$,  
such that $a_2b_3=a_3b_2$. Then there are independent random variables 
$\xi_1$, $\xi_2$,  $\xi_3$, and $\xi_4$  
  with values in $\mathbb{Z}(p)$ with nonvanishing 
  characteristic functions such that the following is true:
\renewcommand{\labelenumi}{\rm(\Roman{enumi})}
\begin{enumerate} 
\item
the random variables $\xi_2$ and  $\xi_3$ are identically distributed; 
 \item
the distribution 
of the random vector $(L_1, L_2)$, where $L_1=\xi_1+a_2\xi_2+a_3\xi_3$ and $L_2=b_2\xi_2+b_3\xi_3+\xi_4$,  need not necessarily 
determine the distribution of the random variables
$\xi_2$ and $\xi_3$ up to  shift. 
\end{enumerate}
\end{proposition}

\begin{proof} 
To prove the proposition, we argue for the field 
$\mathbb{Z}(p)$ in the same
way as in Proposition \ref{pr1} we argued for the field $\mathbb{Q}_p$
and keep the same notation.
The only difference is that instead of (\ref{9.4}) the function $f(y)$ is 
of the form
\begin{equation}\label{08_04_4}
f(y)=
\begin{cases}
1& \text{\ if\ }\ \ y=0,
\\-1& \text{\ if\ }\ \ y\ne 0,
\end{cases}
\end{equation} 
because we construct the distributions $\mu$ and $\nu$ at once on $\mathbb{Z}(p)$.
The proposition will be proved if we verify that that the function $f(y)$ 
satisfies the equation 
$$
f(a_2y)f(a_3y)=1, \quad y\in \mathbb{Z}(p).
$$
  In view of
(\ref{08_04_4}), it is obvious.
\end{proof}

Proposition \ref{pr2} shows that, unlike Theorems \ref{th3} 
and \ref{th2}, in the field
$\mathbb{Z}(p)$ do not exist nonzero elements $a_j$, $b_j$, $j=2,3$,   
 such that $a_2b_3=a_3b_2$ and the distribution of 
the random vector $(L_1, L_2)$  determines 
  the 
distributions of the random variables 
$\xi_2$ and up $\xi_3$ to shift. 
  
\medskip

Let $\mathbb{Q}$ be the field of rational  numbers considering in the discrete
topology. When we say the group $\mathbb{Q}$, we mean the additive 
group of the field $\mathbb{Q}$.
The character group of the group $\mathbb{Q}$ is
topologically isomorphic to  the $\bm a$-adic solenoid
$\Sigma_{\bm a}$, where ${\bm a}=(2, 3, 4, \dots)$ (\!\!\cite[(25.4)]{HeRo1}). 
The group $\Sigma_{\bm a}$ is compact. 
Since the multiplication by
any nonzero integer is an automorphism of the group $\mathbb{Q}$,
the multiplication by
any nonzero integer is a topological automorphism of the group
$\Sigma_{\bm a}$.
Therefore, the multiplication by any nonzero rational number is 
well-defined and is also a topological automorphism in the group $\Sigma_{\bm a}$.
 Note that $(ax, y)=(x, ay)$ 
for all $a, x \in \mathbb{Q}$, $y\in \Sigma_{\bm a}$. If $\mu$ is a distribution
on the group $\mathbb{Q}$, the characteristic function $\hat\mu(y)$ is defined by
formula (\ref{e1}), where $X=\mathbb{Q}$, $Y=\Sigma_{\bm a}$. 
The set $\mathbb{Q}$ is a countable subset of $\mathbb{R}$. For this reason, 
if a random variable $\xi$
take values in $\mathbb{Q}$ we can consider $\xi$ as a random variable 
with values in $\mathbb{R}$. This implies, in particular, that Theorem \ref{th3}
is valid for the field $\mathbb{Q}$, i.e.,
when $\xi_j$ take values in $\mathbb{Q}$ and 
$a_j, b_j\in\mathbb{Q}$. However, the fact that random variables $\xi_j$ 
can be considered as random variables taking values in the discrete 
field $\mathbb{Q}$ allows us to strengthen Theorem \ref{th3}
for the field $\mathbb{Q}$.
\begin{theorem}\label{th5}
Let $\xi_1$,  $\xi_2$,  $\xi_3$, and 
$\xi_4$ be independent
random variables with  values in the field $\mathbb{Q}$ with
nonvanishing characteristic functions. Let $a_j$, $b_j$, $j=2,3$,   
be nonzero 
elements of $\mathbb{Q}$. Consider the linear forms 
$L_1=\xi_1+a_2\xi_2+a_3\xi_3$ and $L_2=b_2\xi_2+b_3\xi_3+\xi_4$. 
Assume that one of the following conditions holds:
\renewcommand{\labelenumi}{\rm(\Roman{enumi})}
\begin{enumerate} 
\item
 $a_2b_3\ne a_3b_2$; 
 \item
$a_2b_3=a_3b_2$, $|a_2| \ne |a_3|$, and   the random variables
$\xi_2$ and  $\xi_3$ are identically  distributed;
\item
$a_2b_3=a_3b_2$, $a_2=a_3$, and   the random variables
$\xi_2$ and  $\xi_3$ are identically  distributed.
\end{enumerate} 
Then  the distribution of 
the random vector $(L_1, L_2)$  determines 
  the 
distributions of the random variables 
$\xi_j$, $j=1, 2, 3, 4$, up to shift. 
\end{theorem}
\begin{proof}
Assume that condition (I) holds. The proof of Theorem
 \ref{th2} in the case, when condition (I) holds, is based only on the following properties of the group $\mathbb{Q}_p$: the group $\mathbb{Q}_p$ consists of compact elements and the mapping $y\rightarrow 2y$ is a topological 
automorphism of the group  $\mathbb{Q}_p$. 
Both of these properties are 
also valid for the group $\Sigma_{\bm a}$, where ${\bm a}=(2, 3, 4, \dots)$. 
For this reason  the proof  
 remains unchanged if, instead of the field  $\mathbb{Q}_p$, 
 we consider the filed $\mathbb{Q}$.  
 
Note that if $a_2b_3\ne a_3b_2$, then it follows from Theorem \ref{th3}  
 that the random vector $(L_1, L_2)$  determines 
  the distributions of the random variables 
$\xi_j$, $j=1, 2, 3, 4$, up to shift, only if additional conditions are satisfied, namely $a_2b_2\ne -a_3b_3$ and the random variables
$\xi_2$ and  $\xi_3$ are identically  distributed.

In the case, when condition (II) or (III) is satisfied, the 
corresponding statements
follows   from 
the corresponding statements of Theorem \ref{th3} 
for the field $\mathbb{R}$.
\end{proof}

\begin{remark}\label{re2}
The distributions $\mu$ and $\nu$ constructed in Remark \ref{re1}
in fact are the distributions on the group of integers. We can
consider $\mu$ and $\nu$ as as distributions on $\mathbb{Q}$.
Hence Remark \ref{re1} is valid for the random variables 
with values in the field $\mathbb{Q}$ and it
shows that we can not strengthen Theorem \ref{th5} 
and replace conditions (II) and (III) 
in Theorem \ref{th5} by the condition $a_2b_3=a_3b_2$.
\end{remark} 

In Theorems \ref{th3}, \ref{th2}, \ref{th4}, and \ref{th5}  
independent random variables take values in 
a locally compact field, namely in $\mathbb{R}$, $\mathbb{Q}_p$,   
$\mathbb{Z}(p)$, and $\mathbb{Q}$. In doing so, coefficients of the linear 
forms are
elements of the field. We can also study a more general 
problem, when independent random variables take values in 
a locally compact Abelian group $X$, and coefficients of the forms
are continuous endomorphisms of $X$.
Taking this into account, we formulate the following problem. 

\textit{Let $X$ be a second countable locally compact Abelian group. 
Let $\xi_1$,  $\xi_2$,  $\xi_3$, and $\xi_4$ be independent
random variables with  values in $X$ with
nonvanishing characteristic functions. Let $a_j$, $b_j$, $j=2,3$,   
be continuous endomorphisms of the group $X$. Consider the linear forms 
$L_1=\xi_1+a_2\xi_2+a_3\xi_3$ and $L_2=b_2\xi_2+b_3\xi_3+\xi_4$.
What are the conditions on $a_j$, $b_j$, and $\xi_j$ to guarantee that
  the distribution of 
the random vector $(L_1, L_2)$ determines 
  the 
distributions of the random variables 
$\xi_j$, $j=1, 2, 3, 4$, up to  shift}?

It follows from the results of the article that these  conditions depend  
on the group $X$.

A more general problem can also be formulated. 

\textit{Let us assume that we know the distribution of the random vector $(L_1, L_2)$. 
How uniquely does this distribution determine the distributions of the random variables
$\xi_j$, $j=1, 2, 3, 4$}?

In connection with this problem, we note that in the case when $X$ is an
arbitrary $\bm a$-adic solenoid $\Sigma_{\bm a}$, it follows from Theorem 3.1 
in \cite{F}, see also  \cite[Theorem 15.8]{F1}, that if 
$L_1=\xi_1+\xi_2+\xi_3$ and $b_2$, $b_3$, and $b_2-b_3$ are topological automorphisms of 
the group $\Sigma_{\bm a}$, then the distribution of 
the random vector $(L_1, L_2)$ determines 
 the 
distributions of the random variables 
$\xi_j$, $j=1, 2, 3, 4$, up to convolution with a Gaussian 
distribution on $\Sigma_{\bm a}$.

\medskip

\noindent{\bf Acknowledgements} 

\medskip

\noindent This article was written during the author's stay at 
the Department of Mathematics University of Toronto as a Visiting Professor. 
I am very grateful to Ilia Binder for his invitation and support.

\medskip

\noindent\textbf{Funding} The author  declares that no funds, grants, 
or other support were
received during the preparation of this manuscript.

\medskip

\noindent\textbf{Data Availability Statement} Data sharing 
not applicable to this article as no datasets were
generated or analysed during the current study.

\noindent B. Verkin Institute for Low Temperature Physics and Engineering\\
of the National Academy of Sciences of Ukraine\\
47, Nauky ave, Kharkiv, 61103, Ukraine

\medskip

\noindent Department of Mathematics  
University of Toronto \\
40 St. George Street
Toronto, ON,  M5S 2E4
Canada 

\medskip

\noindent e-mail:    gennadiy\_f@yahoo.co.uk

\end{document}